\RequirePackage[ngerman,english]{babel}
\documentclass[11pt,a4paper]{amsart}
\usepackage[text={15cm,24cm}]{geometry}

\usepackage{amsfonts, amsmath, amssymb, amsthm}
\usepackage{german}
\usepackage[ngerman,english]{babel}
\usepackage{hyperref}
\usepackage{xspace}

\usepackage{booktabs}
\usepackage{bbm}
\usepackage{overpic}
\usepackage{graphicx}

\newcommand{\ZZ}{{\mathbb{Z}}}
\newcommand{\RR}{{\mathbb{R}}}
\newcommand{\KK}{{\mathbb{K}}}

\newcommand{\cM}{{\mathcal{M}}}

\newcommand{\vones}{\mathbbm{1}}
\newcommand{\frp}{\mathfrak{p}}
\newcommand{\frS}{\mathfrak{S}}

\newcommand{\SetOf}[2]{\left\{#1\vphantom{#2}\,\right.\left|\,\vphantom{#1}#2\right\}}
\newcommand{\smallSetOf}[2]{\{#1\,|\,#2\}}

\DeclareMathOperator{\conv}{conv}
\DeclareMathOperator{\pos}{pos}
\DeclareMathOperator{\rank}{rank}
\DeclareMathOperator{\Sym}{Sym}

\DeclareMathOperator{\tconv}{tconv}
\DeclareMathOperator{\val}{val}

\newcommand\cprime{$'$}

\newcommand{\Atightspan}[1]{{\mathcal{T}(#1)}}

\newcommand{\subdiv}{\Sigma}
\newcommand{\Hypersimplex}[2]{{\Delta(#1,#2)}}

\newcommand{\TT}{{\mathbb{T}}}

\DeclareMathOperator{\tdet}{tdet}

\newcommand{\Gr}[2]{{\mathrm{Gr}(#1,#2)}}
\newcommand{\Dr}[2]{{\mathrm{Dr}(#1,#2)}}

{\end{list}}

\theoremstyle{plain}
\newtheorem{theorem}{Theorem}
\newtheorem*{theorem*}{Theorem}
\newtheorem{proposition}[theorem]{Proposition}
\newtheorem{corollary}[theorem]{Corollary}
\newtheorem{lemma}[theorem]{Lemma}
\newtheorem{question}[theorem]{Question}

\theoremstyle{definition}

\newtheorem{example}[theorem]{Example}

\newtheorem{remark}[theorem]{Remark}

\graphicspath{ {mpost/} }

\begin{document}

\title[Dressians and Their Rays]{Dressians, Tropical Grassmannians,\\ and Their Rays}

\author[Herrmann, Joswig and Speyer]{Sven Herrmann \and Michael Joswig \and David Speyer}
\address{Sven Herrmann, School of Computing Sciences, University of East Anglia, Norwich, NR4~7TJ, UK}
\email{sherrmann@mathematik.tu-darmstadt.de}
\address{Michael Joswig, Fachbereich Mathematik, TU Darmstadt, 64289 Darmstadt, Germany}
\email{joswig@mathematik.tu-darmstadt.de}
\address{David Speyer, Department of Mathematics, 530 Church Street, Ann Arbor, MI 48109-1043 USA}
\email{speyer@umich.edu}
\thanks{The first author was supported by a fellowship within the Postdoc-Program of the German
  Academic Exchange Service (DAAD).  The second author was supported by DFG Priority Program 1489
  ``Experimental Methods in Algebra, Geometry and Number Theory''.}

\date{\today}
\begin{abstract}
  The Dressian $\Dr kn$ parametrizes all tropical linear spaces, and it carries a natural fan
  structure as a subfan of the secondary fan of the hypersimplex $\Delta(k,n)$.  We explore the
  combinatorics of the rays of $\Dr kn$, that is, the most degenerate tropical planes, for arbitrary
  $k$ and $n$.  This is related to a new rigidity concept for configurations of $n-k$ points in the
  tropical $(k{-}1)$-torus.  Additional conditions are given for $k=3$.  On the way, we compute the
  entire fan $\Dr38$.
\end{abstract}

\maketitle

\section{Introduction}

The classical Grassmannian parametrizes all linear spaces over a fixed field.  Its tropicalization,
the \emph{tropical Grassmannian}~\cite{MR2071813}, parametrizes those tropical linear spaces which
arise as tropicalizations of ordinary linear spaces over a field of Puiseux series.  Studying the
tropical Grassmannians and related concepts is motivated by questions in algebraic geometry, for
example see \cite{HackingKeelTevelev}, as well as by applications in algorithmic biology, for
example see~\cite{ASCB}.  Here we are aiming at exploring the tropical Grassmannian $\Gr{k}{n}$ from
the combinatorial point of view.  To this end, we study an outer approximation, the \emph{Dressian}
$\Dr{k}{n}$, which is the polyhedral fan of those regular subdivisions of the hypersimplex
$\Delta(k,n)$ which have the property that each cell is a matroid polytope.  In this manner,
$\Dr{k}{n}$ is a subfan of the secondary fan of $\Delta(k,n)$. Both fans have a non-trivial
$n$-dimensional lineality space and we will consider them (and all other fans in this paper) and
their cones always modulo this lineality, meaning that the smallest non-trivial cones (of dimension
$n+1$) are considered to be one-dimensional and called rays.  Alternatively, as a set, the Dressian
can also be described as the tropical pre-variety which arises as the intersection of the tropical
hypersurfaces defined by the $3$-term Pl\"ucker relations; for other characterizations see
Proposition~\ref{prop:matroid_decomposition} below. Therefore, Dressians were called \emph{tropical
  pre-Grassmannians} in~\cite{MR2071813}.  From the description via the $3$-term Pl\"ucker
relations, it follows that $\Dr{k}{n}$ contains $\Gr{k}{n}$ as a set.  Asymptotically, for fixed $k$
and growing $n$ the Dressians $\Dr{k}{n}$ become much larger than the tropical Grassmannians
$\Gr{k}{n}$ \cite[Thm.~3.6]{MR2515769}.

Products of simplices naturally arise in this context as the vertex figures of hypersimplices, and
the relationship between the secondary fan structures of these two polytopes has been studied
previously, for example, by Kapranov~\cite[Sec.~1.4]{MR1237834}.  Develin and Sturmfels
\cite[Thm.~1]{DevelinSturmfels04} showed that the regular subdivisions of products of simplices are
dual to tropical convex hulls of finitely many points.  Our first main result is
Corollary~\ref{cor:embedding}:

\begin{theorem*}
  For $n>k>1$ there is a piecewise-linear embedding $\tau$ of the secondary fan of the product of
  simplices $\Delta_{k-1}\times\Delta_{n-k-1}$ into the Dressian $\Dr{k}{n}$.  The image of $\tau$
  is contained in the tropical Grassmannian $\Gr{k}{n}$ as a set.
\end{theorem*}

The secondary fan of $\Delta_{k-1}\times\Delta_{n-k-1}$ has the same dimension $nk-n-k^2+1$ as
$\Gr{k}{n}$ (modulo lineality), and $\tau$ is a homeomorphism onto its image. For each of the
$\binom{n}{k}$ vertices of $\Delta(k,n)$, we get an inner approximation of the tropical
Grassmannians in this manner.  Via the same approach we also relate non-regular subdivisions of
products of simplices to non-regular matroid subdivisions of hypersimplices.  After making a
preprint version of our result available we learned that Rinc\'on independently proved a similar
statement~\cite{Rincon}.

The goal of the present paper is to combinatorially describe the Dressians as well as possible.
Since, in a way, the Dressians encode all of matroid theory this is quite an endeavor.  Our
expectations must therefore be modest.  So we focus on the rays of $\Dr{k}{n}$, that is, on those
tropical linear spaces corresponding to matroid decompositions of $\Delta(k,n)$ which can only be
coarsened in a trivial way.  We call a configuration of $k$ points in the tropical torus $\TT^{n-1}$
\emph{tropically rigid} if it does correspond to a ray of the secondary fan of
$\Delta_{k-1}\times\Delta_{n-1}$.  Via the aforementioned Corollary~\ref{cor:embedding}, tropically
rigid point configurations give rise to rays of the Dressians which are also rays of the tropical
Grassmannians.  The known coarsest matroid subdivisions of the hypersimplices are the \emph{splits}
(with precisely two maximal cells) \cite[Prop.~5.2]{MR2502496} and the \emph{$3$-splits} (with
precisely three maximal cells sharing a common codimension-$2$ cell) \cite[Cor.~6.4]{Herrmann09}.
It turns out that all coarsest matroid subdivisions of $\Delta(3,8)$ but one (up to symmetry) are
induced by coarsest subdivisions of a vertex figure; see Figures~\ref{fig:rigid-special},
\ref{fig:trop-rigid} and~\ref{fig:non-planar} below.

Our second main result is the explicit computation of the entire fan $\Dr{3}{8}$ via
\texttt{poly\-make}~\cite{DMV:polymake}.  This computation, in particular, leads to our
Theorem~\ref{thm:polymake}:

\begin{theorem*}
  The Dressian $\Dr38$ is a non-pure non-simplicial nine-dimensional polyhedral fan with $f$-vector
  \begin{equation*}
    \begin{aligned}
      (1;&\,15,\!470;\,642,\!677;\,8,\!892,\!898 ;\,57,\!394,\!505 ;\,194,\!258,\!750
      ; \\
      &\,353,\!149,\!650 ;\,324,\!404,\!880;\,117,\!594,\!645 ;\,113,\!400 )\,.
    \end{aligned}
  \end{equation*}
  Modulo the natural $\Sym(8)$-symmetry, the $f$-vector reads
  \[
  (1;\,12;\,155;\, 1,\!149;\, 5,\!013;\,12,\!737 ;\,18,\!802 ;\,14,\!727 ;\,4,\!788;\,14)\,.
  \]
  There are $116,\!962,\!265$ maximal cones, $113,\!400$ of dimension $9$ and
  $116,\!848,\!865$ of dimension $8$. Up to symmetry, there are $4,\!748$ maximal cones,
  $14$ of dimension $9$ and $4,\!734$ of dimension $8$.
\end{theorem*}

We are indebted to an anonymous referee for several very helpful suggestions.

\section{Tropical Polytopes and Matroid Subdivisions} \label{sec:Polytope}

A map $\pi$ from $\tbinom{[n]}{k}$, the set of $k$-subsets of $[n]=\{1,2,\dots,n\}$, to the set
$\RR$ is called a \emph{(finite) tropical Pl\"ucker vector} if the minimum of the three numbers
\begin{equation}\label{eq:3term}
  \pi(\rho ij)+\pi(\rho\ell m) \, , \quad \pi(\rho i\ell)+\pi(\rho jm) \, , \quad \pi(\rho
  im)+\pi(\rho j\ell)
\end{equation}
is attained at least twice for each choice $\rho$ of a $(k{-}2)$-subset of $[n]$ and pairwise
distinct $i,j,\ell,m\in[n]\setminus\rho$; here we use the common shorthand notation $\rho ij$ for
the set $\rho\cup\{i,j\}$.  Condition~\eqref{eq:3term} is equivalent to requiring that $\pi$ is
contained in the tropical pre-variety which arises as the intersection of the tropical hypersurfaces
of all $3$-term Pl\"ucker relations~\cite{MR2071813}; this tropical pre-variety is the
\emph{Dressian} $\Dr kn$.  Throughout this paper we assume that $n>k>0$.

A particularly interesting class of finite tropical Pl\"ucker vectors comes about as
follows. Consider a matrix $V\in\RR^{k\times (n-k)}$.  The \emph{augmented matrix} of $V$ is the
$k{\times}n$-matrix $\bar V=(E_k|V)$, where $E_k$ is the \emph{tropical identity matrix} of rank $k$
(the $k{\times}k$-matrix with $0$ on the diagonal and coefficients equal to $\infty$ otherwise), and
$(A|B)$ denotes the block column matrix formed from the columns of $A$ and $B$.  Each $k$-element
subset $\sigma\subseteq[n]$ specifies a $k{\times}k$-submatrix $\bar V_\sigma$ by selecting the
columns of $\bar V$ whose indices are in $\sigma$.  Now the map
\begin{equation}
  \label{eq:tau_v}
  \tau_V \,:\, \binom{[n]}{k}\to\RR \ , \quad \sigma \mapsto \tdet(\bar V_\sigma)
\end{equation}
is a finite tropical Pl\"ucker vector.  Here
\begin{equation}\label{eq:tdet}
  \tdet(A) \ = \ \min_{\omega\in\Sym(k)} a_{1,\omega(1)}+a_{2,\omega(2)}+\dots+a_{k,\omega(k)}
\end{equation}
denotes the \emph{tropical determinant} of the matrix $A=(a_{ij})_{i,j}\in\RR^{k\times k}$, and
$\Sym(k)$ is the symmetric group naturally acting on the set $[k]$.  We obtain a map $\tau$ which
sends the $k{\times}(n{-}k)$-matrix~$V$ to the vector $\tau_V$ of length $\tbinom{n}{k}$.
Conversely, for $\pi\in\RR^{\tbinom{[n]}{k}}$ we define a $k{\times}(n{-}k)$-matrix
$\Phi(\pi)=(\phi_{ij})_{i,j}$ by letting
\[
\phi_{ij} \ = \ \pi(([k]\setminus\{i\})\cup\{j+k\}) \, .
\]
That is, the coefficients of $\Phi(\pi)$ are the tropical determinants of those
$k{\times}k$-submatrices of $\bar V$ which are formed by the first $k$ columns, except for the
$i$-th which is replaced by the $(j{+}k)$-th column of $\bar V$, and which is the same as the $j$-th
column of $V$.

\begin{lemma}\label{lem:Phi-tau}
  For an arbitrary matrix $V\in\RR^{k\times (n-k)}$ we have $\Phi(\tau_V) = V$.  In particular, the
  map $\tau$ is injective, and the map $\Phi$ is surjective onto $\RR^{k\times (n-k)}$.  Moreover, $\Phi$ is linear and $\tau$ is piecewise-linear.
\end{lemma}

\begin{proof}
  For $\sigma=([k]\setminus\{i\})\cup\{j+k\}$ the matrix $\bar V_\sigma$ has precisely one column
  with finite entries only, namely the last one, corresponding to column $j$ of the matrix $V$.
  Each of the first $k-1$ columns of $\bar V_\sigma$ has precisely one finite entry, which equals
  zero in all cases.  Among the first $k-1$ columns the only row with only $\infty$ coefficients is
  the $i$-th one.  Hence the tropical determinant of~$\bar V_\sigma$ equals the coefficient $v_{ij}$
  of the matrix $V$.  This is precisely the first claim.

  The tropical determinant is a piecewise-linear map; see \eqref{eq:tdet}.  Hence $\tau$ is piecewise-linear, too.  The map $\Phi$ is a linear projection.
\end{proof}

In general, the map $\tau$ is not surjective; for details see Example~\ref{exmp:k=2} below.

\begin{remark}
  The map $\tau$ is not linear: For instance, consider $k=2$, $n=4$,
  \[
  V \ = \ \begin{pmatrix} 1 & 0 \\ 0 & 1 \end{pmatrix} \, ,
  \quad \text{and} \quad
  W \ = \ \begin{pmatrix} 0 & 1 \\ 1 & 0 \end{pmatrix} \, .
  \]
  Then $\tau_V=(0,1,0,0,1,0)$, $\tau_W=(0,0,1,1,0,0)$, and
  \[
  \tau_{V+W} \ = \ (0,1,1,1,1,2) \ \ne \ (0,1,1,1,1,0) \ = \ \tau_V+\tau_W \, ,
  \]
  where the coordinates of $\RR^{\tbinom{4}{2}}$ are (lexicographically) labeled $12$, $13$, $14$,
  $23$, $24$, $34$.
\end{remark}

The map $\tau_V$ can be read as a height function on the hypersimplex
\[
\Hypersimplex{k}{n} \ := \ \conv\SetOf{e_\sigma:=\sum_{i\in \sigma} e_i\in \RR^n}{\sigma\in\binom{[n]}{k}}
\]
via mapping the vertex $e_\sigma$ to $\tau_V(\sigma)$.  Similarly, the matrix $V$ has a natural
interpretation as a height function on the vertices of $\Delta_{k-1}\times\Delta_{n-k-1}$.

A \emph{subpolytope} of a polytope with vertex set $X$ is the convex hull of a subset of $X$. A
subpolytope of $\Delta(k,n)$ whose edges are parallel to edges of $\Delta(k,n)$, that is,
differences of standard basis vectors $e_i-e_j$, is a \emph{$(k,n)$-matroid polytope}.  By a result
of Gel{\cprime}fand, Goresky, MacPherson and Serganova~\cite[Thm.~4.1]{GelfandEtAl87} the 
vertices of a $(k,n)$-matroid
polytope correspond to the bases of a matroid of rank at most $k$ on $n$ elements; hence the name
\emph{matroid polytope}.  A \emph{$(k,n)$-matroid subdivision} is a polytopal subdivision (of
$\Delta(k,n)$) such that each cell is a $(k,n)$-matroid polytope.  A \emph{regular} (matroid)
subdivision is induced by a height function; see \cite[Sec.~2.2.3]{Triangulations} for details about
regular subdivisions.  The following characterization is crucial.

\begin{proposition}[{Kapranov~\cite{MR1237834} and \cite[Prop.~2.2]{TLS}}]
  \label{prop:matroid_decomposition}
  Let $\Sigma$ be a regular subdivision of $\Delta(k,n)$ induced by the lifting function $\pi$.  Then
  the following are equivalent.
  \begin{enumerate}
  \item $\pi$ is a finite tropical Pl\"ucker vector,
  \item $\Sigma$ is a matroid subdivision,
  \item \label{it:matroid_decomposition:edges} the $1$-cells of $\Sigma$ are precisely the edges of
    $\Delta(k,n)$.
  \end{enumerate}
\end{proposition}

For $\sigma,\sigma'\in\tbinom{[n]}{k}$ the vertices $e_\sigma$ and $e_{\sigma'}$ are neighbors in
the vertex-edge graph of $\Delta(k,n)$ if and only if the symmetric difference of $\sigma$ and
$\sigma'$ consists of two elements.  This means that the neighbors of $e_\sigma$ are contained in
the set
\[
e_\sigma^\perp \ := \ \SetOf{x \in \RR^n}{\sum_{i=1}^n x_i=k \text{ and } \sum_{i\in\sigma} x_i=k-1}\,,
\]
which forms a hyperplane in the affine span of the hypersimplex.  The vertex figure of each vertex
in $\Delta(k,n)$ is isomorphic to the product of simplices $\Delta_{k-1}\times\Delta_{n-k-1}$.  The map
\[
\iota \,:\, \SetOf{e_i+e_j}{i\in[k],j\in[n]\setminus [k]} \to e_{[k]}^\perp \,, \ e_i+e_j\mapsto e_{[k]\setminus\{i\}\cup\{j\}}
\]
describes a bijection from the vertices of $\Delta_{k-1}\times\Delta_{n-k-1}$ to the neighbors of
the vertex $e_{[k]}$ of $\Delta(k,n)$. This naturally induces a map form the set of all subdivisions of
$\Delta_{k-1}\times\Delta_{n-k-1}$ to the set of all subdivisions of the vertex figure $\Hypersimplex kn \cap
e_{[k]}^\perp$, which we also denote by $\iota$. For any not necessarily regular subdivision $\Sigma$ of
$\Delta(k,n)$ we let $\Sigma\cap e_{[k]}^\perp$ denote the polytopal complex arising from intersecting
each cell of $\Sigma$ with the affine subspace $e_{[k]}^\perp$.

The following is the content of Corollary 1.4.14 in \cite{MR1237834}; here we give an elementary proof.

\begin{proposition}\label{prop:vertex_figure}
  Let $V\in\RR^{k\times (n-k)}$ be a matrix, $\Gamma$ the regular subdivision of
  $\Delta_{k-1}\times\Delta_{n-k-1}$ induced by $V$, and $\Sigma$ the regular matroid subdivision of
  $\Delta(k,n)$ induced by $\tau_V$.

  Then the polytopal complex $\Sigma\cap e_{[k]}^\perp$ coincides with $\iota(\Gamma)$.
\end{proposition}

\begin{proof}
  The neighbors of the vertex $e_{[k]}$ lie in the common hyperplane $e_{[k]}^\perp$, and therefore
  the set $\Delta(k,n)\cap e_{[k]}^\perp$ serves as a model for the vertex figure of $e_{[k]}$.
  Since $\Sigma$ is a matroid subdivision,
  Proposition~\ref{prop:matroid_decomposition}~\eqref{it:matroid_decomposition:edges} implies that
  each edge of $\Sigma$ intersects the hyperplane $e_{[k]}^\perp$ in a vertex. This says that
  $\Sigma\cap e_{[k]}$ is a (regular) subdivision of the vertex figure $\Hypersimplex kn \cap
  e_{[k]}^\perp$ (without any new vertices).  This way the claim follows
  from Lemma~\ref{lem:Phi-tau}.
\end{proof}

The inclusion relation among the cells turns a pure polytopal complex $\Sigma$ of dimension $m$ into
a partially ordered set.  The dimension of a cell serves as a rank function on the poset: The poset
elements of rank $\ell+1$ are the $\ell$-dimensional cells of $\Sigma$.  The \emph{tight-span}
$\Atightspan{\Sigma}$ is the partially ordered set obtained by restricting the previously mentioned
partial order to interior cells, and dualizing it.  So the elements of the tight-span rank $\ell+1$
are the interior cells of dimension $m-\ell$, and the partial ordering is given by reverse
inclusion.  If $\Sigma$ is a polytopal subdivision of a polytope, the tight-span is isomorphic to
the containment poset of a contractible cell complex.

If, additionally, this subdivision is regular, then this cell complex can be realized as a polytopal
complex. Obviously, if $\Sigma$ is a subdivision of a polytope $P$ and $\Sigma'$ a refinement of
$\Sigma$, then $\Atightspan\Sigma$ is a subcomplex of $\Atightspan{\Sigma'}$. Furthermore, if $P'$
is a subpolytope of $P$, the tight-span of the subdivision $\Sigma|_{P'}:=\smallSetOf{S\cap
  P'}{S\in\Sigma}$ is a subcomplex of $T(\Sigma)$.  In general, $\Sigma|_{P'}$ may have vertices
which do not occur in $\Sigma$; see Figure~\ref{fig:subdivisions} for an example.

\begin{proposition}\label{prop:ts-isom}
  Let $P$ be a polytope, $P'$ a subpolytope of $P$ and $\Sigma$ a subdivision of $P$ such that
  $T(\Sigma)$ and $T(\Sigma|_{P'})$ are isomorphic.
  \begin{enumerate}
  \item\label{prop:ts-isom:isom} If $\Sigma'$ is a subdivision of $P$ coarsening $\Sigma$, then $T(\Sigma')$ and 
    $T(\Sigma'|_{P'})$ are isomorphic.
  \item\label{prop:ts-isom:coarsest} If $\Sigma|_{P'}$ is a coarsest subdivision of $P'$, then $\Sigma$
    is a coarsest subdivision of $P$.
  \end{enumerate}
\end{proposition}

\begin{figure}
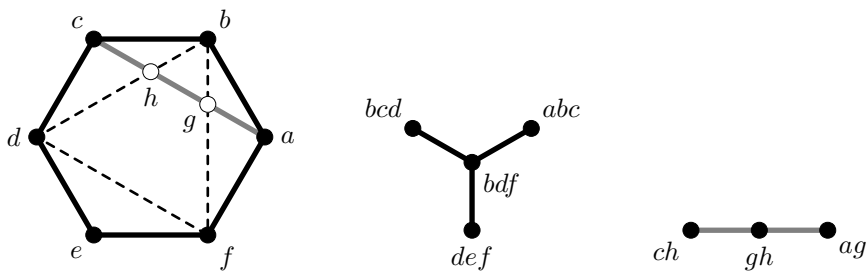

  \includegraphics[scale=1]{subdivisions1.mps} \qquad
  \includegraphics[scale=1]{subdivisions2.mps} \qquad
  \includegraphics[scale=1]{subdivisions3.mps}
  \caption{Subdivision $\Sigma$ of the hexagon $P=abcbdef$ (left).  The induced subdivision
    $\Sigma|_{P'}$ of the subpolytope $P'=ac$ has two new vertices, labeled $g$ and $h$.  Tight span
    of $\Sigma$ (center) and tight span of $\Sigma|_{P'}$ (right)}
  \label{fig:subdivisions}
\end{figure}

\begin{proof}
  We first prove \eqref{prop:ts-isom:isom}.  Suppose $T(\Sigma')$ and $T(\Sigma'|_{P'})$ are not
  isomorphic, that is, there exists a full-dimensional cell $F'\in\Sigma'$ such that $F'\cap P'$ is
  not full-dimensional. However, since $T(\Sigma)$ and $T(\Sigma|_{P'})$ are isomorphic, for
  a full"=dimensional cell $F\subset F'$ of $\Sigma$, we have that $F\cap P'\subset F'\cap P'$ is
  full"=dimensional, a contradiction.

  Now suppose that there exists a non"=trivial coarsening $\Sigma'$ of
  $\Sigma$. By~\eqref{prop:ts-isom:isom}, $T(\Sigma')$ and $T(\Sigma'|_{P'})$ are isomorphic. This
  shows that $\Sigma'|_{P'}$ is a non"=trivial coarsening of $\Sigma|_{P'}$, and this establishes
  \eqref{prop:ts-isom:coarsest}.
\end{proof}

The following general uniqueness result applies to our main result below.

\begin{proposition}\label{prop:unique}
  Let $P$ be a polytope and $v$ a vertex of $P$. Further, let $N$ be the set of vertices neighboring
  $v$, and let $Q = \conv N$.  Suppose that $\dim Q = \dim P - 1$.  Then for each subdivision
  $\Gamma$ of $Q$ there is at most one subdivision $\Sigma$ of $P$ such that $\Sigma\cap Q =
  \Gamma$ and every maximal cell of $\Sigma$ contains $v$.
\end{proposition}

\begin{proof}
  For each cell $\gamma$ of $\Gamma$, let $C(\gamma)$ be the cone from $v$ over $\gamma$. Let
  $D(\gamma) = C(\gamma) \cap P$. Then $P$ is stratified into the $D(\gamma)$, but this may not be a
  polyhedral decomposition because $D(\gamma)$ may have vertices which are not vertices of $P$.  We
  will show that, if there is any subdivision $\Sigma$ with the required properties, then
  $\Sigma=\smallSetOf{D(\gamma)}{\gamma\in\Gamma}$.

  First, let $\sigma$ be any maximal cell of $\Sigma$. Then $\sigma \cap Q$, by dimensionality, is a
  maximal cell of $\Gamma$; call it $\gamma(\sigma)$. Since $\sigma$ contains $v$, we have that
  $\sigma$ contains $\conv (v,\gamma(\sigma))$.  This means that there cannot be a second maximal
  cell $\sigma'\ne\sigma$ with $\gamma(\sigma')=\gamma(\sigma)$, as then they would overlap in the
  full dimensional set $\conv(v, \gamma(\sigma))$.

  Also, $\sigma$ has a vertex at $v$, and has no other vertices lying on the $v$-side of $Q$ (since there
  are no such vertices in $P$).  So $\sigma$ is contained in $C(\gamma(\sigma))$ and is thus contained in
  $C(\gamma(\sigma)) \cap P = D(\gamma(\sigma))$.

  Thus, we have shown that there is an injection from facets of $\Sigma$, to facets of $\Gamma$,
  such that $\sigma \subseteq D(\gamma(\sigma))$. But, since $\Sigma$ is supposed to be a
  decomposition of all of $P$, we must have equality for every $\sigma$. As promised, we have shown
  that $\smallSetOf{D(\gamma)}{\gamma\in\Gamma}$ is the only decomposition with the properties
  required.
\end{proof}

The subsequent result is of key relevance.  Rinc\'on independently proved a similar result for
regular subdivisions~\cite{Rincon}.

\begin{theorem}\label{thm:tight-span}
  Let $\Gamma$ be a not necessarily regular subdivision of $\Delta_{k-1}\times\Delta_{n-k-1}$.  Then
  there exists a subdivision $\Sigma$ of $\Delta(k,n)$ such that:
  \begin{enumerate}
  \item $\Sigma\cap e_{[k]}^\perp=\iota(\Gamma)$.\label{it:tight-span:equal}
  \item Each maximal cell of $\Sigma$ contains the vertex $e_{[k]}$. \label{it:tight-span:max-cell}
  \item The tight-spans $\Atightspan{\Gamma}$ and $\Atightspan{\Sigma}$ are isomorphic. \label{it:tight-span:iso}
  \end{enumerate}
  Furthermore, if $\Gamma$ is induced by a lifting function $V\in\RR^{k\times(n-k)}$ and thus
  regular, then $\Sigma$ is also regular and induced by~$\tau_V$.  If, however, $\Gamma$ is not
  regular, then $\Sigma$ is not regular either.
\end{theorem}

In particular, each tight-span of a (regular) subdivision of $\Delta_{k-1}\times\Delta_{n-k-1}$ also
arises as the tight-span of a (regular) matroid subdivision of $\Delta(k,n)$.  Notice that
Theorem~\ref{thm:tight-span} also yields an independent proof of
Proposition~\ref{prop:vertex_figure}.  It follows from Proposition~\ref{prop:unique} that $\Sigma$
is uniquely determined by $\Gamma$.

\begin{proof}
  For $S$ a subset of a real vector space, write $\pos S$ for the positive real span of $S$.

  Consider the polytope $T:=\conv{\smallSetOf{e_\sigma-e_{[k]}}{\sigma\in\binom{[n]}k}}$ obtained by
  translating $\Delta(k,n)$.  We equip the vertex figure $U$ at the origin with the given subdivision
  $\Gamma$ of $\Delta_{k-1}\times\Delta_{n-k-1}$ and define
  \begin{equation}\label{eq:Sigma}
    \Sigma \ := \ \SetOf{\pos \gamma \cap T}{\gamma \in \Gamma}\,.
  \end{equation}
  Notice that $\Sigma$ is constructed as in the proof of Proposition~\ref{prop:unique}.

  We will show that $\Sigma$ is a valid subdivision of $T$.  By construction and since $\Gamma$ is a
  valid subdivision of $U$, it suffices to show that all zero-dimensional faces of $\Sigma$ are
  actually vertices of $T$.  So let $v\in \Sigma$ be zero-dimensional.  Then $v$ is the intersection
  of linear hyperplanes spanned by vertices of $U$ and a face $F$ of $T$.  The vertices of $U$ are
  all points of the form $e_i-e_j$ with $j\in[k]$ and $i\in [n]\setminus [k]$. Taking into account
  that $U$ is contained in the hyperplane $\sum_{i=1}^n x_i=0$, this implies that a hyperplane
  spanned by these vertices can be described by an equation of the form $\sum_{i \in A\cup B}x_i=0$
  for some non-empty $A\subset [k]$ and $B\subset [n]\setminus [k]$.  On the other hand, $F$ (as a
  face of $T$) is the intersection of hyperplanes of the form $x_i\in\{0,1,-1\}$. The only
  possibility for an intersection of these two types of hyperplanes to be zero-dimensional is to be
  a vertex of $U$.  This shows \eqref{it:tight-span:equal}.

  It is immediate from the construction that each maximal cell of $\Sigma$ contains the origin.
  This establishes \eqref{it:tight-span:max-cell}.  To show \eqref{it:tight-span:iso} first observe
  that for any two cells $C$ and $D$ of $\Gamma$ we have $(\pos C\cap T)\cap(\pos D\cap T)=(\pos
  C\cap D)\cap T$.  Second, $\dim(\pos C\cap T)=\dim C+1$, and hence maximal cells in $\Gamma$
  correspond to maximal cells in $\Sigma$.  This yields \eqref{it:tight-span:iso}.

  We now turn to the situation where $\Gamma$ is regular and induced by the lifting function $V$.
  By construction, this implies that $\Sigma$ then is induced by the lifting function $\kappa$ with
  \begin{align}
    \label{eq:sigma-lifting}
    \kappa(e_\sigma-e_{[k]}) \ = \ \min_w \sum_{i,j} w(i,j) \bar V_{i,j}\,,
  \end{align}
  where $w$ ranges over all functions $[k] \times [k+1,n] \to \RR_{\geq 0}$ with $\sum_{i,j}
  w(i,j)(e_j-e_i) = e_B - e_A$. Here we set $A = [k] \setminus ([k] \cap \sigma)$ and $B = [k+1,n]
  \cap \sigma$ for all $\sigma\in\tbinom{[n]}{k}$.

  On the other hand, we have
  \begin{align}
    \label{eq:tau-V-lifting}
    \tau_V(\sigma) \ = \ \tdet(\bar V_\sigma) \ = \ \min_\alpha \sum_{i\in A} \bar V_{i,\alpha(i)}\,,
  \end{align}
  where the minimum ranges over all bijections $\alpha: A\to B$.  Note that $\#A=\#B$ by
  construction.

  The minimum in \eqref{eq:sigma-lifting} is obtained by the minimal cost flow in the complete
  bipartite graph with vertex set $A\cup B$ and edge set $\smallSetOf{\{i,j\}}{i\in A,j\in B}$ with
  one unit of fluid coming in at each of the sources in $A$ and going out at each of the sinks in $B$;
  where $\bar V_{ij}$ is the cost of flowing one unit from $i$ to $j$.

  The minimum in \eqref{eq:tau-V-lifting} is the minimal cost flow in the same graph but restricting
  the flow values to $0$ and $1$.  Solutions to minimal cost flow problems with integral constraints
  are always integral, showing that $\kappa(e_\sigma-e_{[k]})=\tau_V(\sigma)$; see, e.g.,
  \cite[\S10.2]{Schrijver03}.

  It remains to consider the situation when $\Gamma$ is not regular.  Suppose $\Sigma$ were regular
  with lifting function $\lambda$.  Then by restricting $\lambda$ to the vertices of
  $\Hypersimplex{k}{n}$ which are neighbors to the origin, we would obtain a regular subdivision of
  the vertex figure $U$.  By construction this would agree with $\Gamma$, a contradiction.
\end{proof}

\begin{remark}
  Let $P$ be a face of $\Gamma$. Let $G$ be the bipartite graph with vertex set $[n]$ and with an
  edge $(i,j)$ if $e_{[k]} - e_i + e_j$ is a vertex of $F$.  Let $Q$ be the corresponding face of
  $\Sigma$. Then the above proof shows that the matroid corresponding to $Q$ is the principal
  transversal matroid of the graph $G$; see \cite{Transversal}.
\end{remark}

The symmetric group $\Sym(n)$ acts linearly on the Euclidean space $\RR^n$ by permuting the
coordinate directions.  This induces a transitive action on the set of vertices of the hypersimplex
$\Delta(k,n)$.  The stabilizer of a vertex acts transitively on the set of neighbors of this vertex.  
This induces a vertex-transitive action on $\Delta_{k-1}\times\Delta_{n-k-1}$.  On the level
of lifting functions written in matrix form, this action permutes the rows and the columns.
Throughout we identify a $k{\times}(n-k)$-matrix with its ordered sequence of $n-k$ points in
$\TT^{k-1}$.  Permuting the columns corresponds to translating the points in the configuration, and
permuting the rows corresponds to the induced action on the $k$ coordinate directions of the
tropical torus $\TT^{k-1}$.  Now two point configurations (of $n-k$ points each) in $\TT^{k-1}$ are
\emph{equivalent} if they are in the same orbit of the semi-direct product $\RR^k\rtimes\Sym(k)$,
where the additive group of $\RR^k$ acts by translations.  In this sense, two
$k{\times}(n-k)$-matrices are equivalent if they can be transformed into one another by the
following operations:
\begin{enumerate}
\item permuting the rows,
\item permuting the columns,
\item adding an arbitrary vector in $\RR^k$ to all the columns,
\item adding a constant multiple of $\vones$ to any column.
\end{enumerate}
Notice that the roles of the rows and the columns in the above is symmetric: adding a constant
multiple of $\vones=(1,1,\dots,1)$ to the $i$-th row is the same as adding the vector $e_i$ to all
columns, and adding the vector $v\in\RR^{n-k}$ to all the rows is the same as adding $v_j\vones$ to
the $j$-th column for all $1\le j\le n-k$.  The action of $\Sym(n)$ on $\RR^n$ also induces natural
actions on the Grassmannian $\Gr kn$ and on the Dressian $\Dr kn$.

In view of the transitivity of the $\Sym(n)$ action on the vertices of $\Delta(k,n)$,
Proposition~\ref{prop:vertex_figure} and Theorem~\ref{thm:tight-span} generalize to arbitrary
vertices $e_\sigma$ of $\Hypersimplex kn$ instead of $e_{[k]}$. In fact, the entire construction in
the beginning of this section can be generalized for an arbitrary $k$-element subset $\sigma$ of
$[n]$ to define functions $\tau^\sigma_V$ and $\Phi^\sigma$ having the same properties as
$\tau_V=\tau^{[k]}_V$ and $\Phi=\Phi^{[k]}$.

Since there is no restriction on the matrix $V$ we obtain the following result.

\begin{corollary}
  For each regular subdivision $\Gamma$ of $\Delta_{k-1}\times\Delta_{n-k-1}$ and for each vertex
  $v$ of $\Delta(k,n)$ there exists a regular matroid subdivision $\Sigma$ of $\Delta(k,n)$ such that
  the subdivision of the vertex figure of $v$ induced by $\Sigma$ coincides with~$\Gamma$.
\end{corollary}

It is now an interesting question to ask which regular matroid subdivisions of $\Delta(k,n)$ are
induced by tropical Pl\"ucker vectors $\tau_V$ as defined in \eqref{eq:tau_v}.  This question has
the following complete answer for $k=2$.

\begin{example}\label{exmp:k=2}
  Let $k=2$.  Then $\Delta_1\times\Delta_{n-3}$ is a prism over a simplex, and the tight-span of any
  of its regular subdivisions is a path.  Equivalently, the tropical convex hull of two distinct
  points $p$ and $q$ in $\TT^{n-3}$ is a one-dimensional polytopal complex.  The tropical line
  segment $\tconv(p,q)$ is the union of at most $n-3$ ordinary line segments.  If $p$ and $q$ are
  generic, equality is attained, and $V:=(p|q)\in\RR^{2\times (n-2)}$ induces a (regular)
  triangulation of $\Delta_1\times\Delta_{n-3}$.  In this case the tight-span of the matroid
  decomposition induced by $\tau_V$ corresponds to a \emph{caterpillar tree} with precisely $n-3$
  interior edges; see \cite[Fig.~8 (left)]{MR2515769}.  In the non-generic case some of these
  interior edges shrink to points.  The construction of $\tau_V$ from $p$ and $q$ is a special case
  of the construction of a tropical linear space from a set of tropically collinear points due to
  Develin~\cite{MR2131129}.  For more details see also Example~\ref{exmp:k=2,cont} below.

  Notice that all trivalent trees with (at most) five leaves are caterpillar trees.  For
  $n~{=}~6$ there is precisely one combinatorial type of tree which is not a caterpillar tree: the
  \emph{snowflake tree}; see \cite[Fig.~8 (right)]{MR2515769}.  For $n>6$ there is a greater
  variety of non-caterpillar trees. These trees correspond to elements of $\Gr 2n$ but are not
  be obtained from tropical line segments via the map $\tau$.
\end{example}

For our investigations further below it is also instructive to look at one more special case.

\begin{example}\label{exmp:k=3,n=6}
  Let $k=3$ and $n=6$.  The tropical Grassmannian $\Gr{3}{6}$ has been analyzed in detail in
  \cite{MR2071813} (see also~\cite[Fig.~1]{MR2515769}). We will adopt their
  notation.  Observe that $\Gr36$ and the Dressian $\Dr36$ have the same support, but the
  fan structures differ.  Up to symmetry, there are seven distinct combinatorial types of finest
  matroid subdivisions of $\Delta(3,6)$, that is, there are seven types of generic $2$-planes in
  $5$-space.  Five of these types arise from configurations of three points in $\TT^2$ via the
  lifting $V\mapsto\tau_V$; we list five $3\times3$-matrices along with the types of the induced
  generic $2$-planes; see Figure~\ref{fig:planes}:
  \[
  \begin{array}{ccccc}
    \begin{pmatrix} 2 & 1 & 0 \\ 0 & 2 & 0 \\ 0 & 0 & 1 \end{pmatrix} &
    \begin{pmatrix} 3 & 0 & 2 \\ 0 & 1 & 0 \\ 0 & 0 & 1 \end{pmatrix} &
    \begin{pmatrix} 0 & 0 & 0 \\ 0 & 1 & 2 \\ 0 & 2 & 4 \end{pmatrix} &
    \begin{pmatrix} 0 & 2 & 2 \\ 0 & 3 & 0 \\ 0 & 0 & 1 \end{pmatrix} &
    \begin{pmatrix} 0 & 1 & 1 \\ 1 & 0 & 1 \\ 1 & 1 & 0 \end{pmatrix} \\[1.75em]
    \text{EEEG} & \text{EEFG} & \text{EEFF(b)} & \text{EFFG} & \text{FFFGG}
  \end{array}
  \]
  The two remaining types EEEE and EEFF(a) do not occur in this way.  Notice that the tight-spans of
  EEFF(a) and EEFF(b) only differ in their labellings; both ``look like'' tropical polytopes, but
  one arises via the lifting $V\mapsto\tau_V$, while the other one does not.
  Figure~\ref{fig:planes} shows the tropical complexes, which coincide with the respective tight
  spans of the induced matroid subdivisions.  The latter are explicitly described in
  Table~\ref{tab:planes}, and this list shows that $e_{123}$ is a vertex of each maximal cell in
  each of these matroid decompositions. All occurring matroids are graphical, defined by one of the two
  graphs in \cite[Fig.~7]{MR2515769} to the left.
\end{example}

\begin{figure}[hbt]
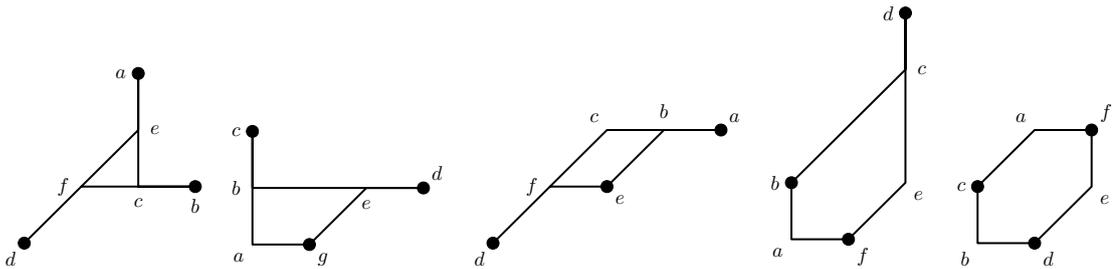

  \includegraphics[scale=.75]{planes1.mps}\quad
  \includegraphics[scale=.75]{planes2.mps}\quad
  \includegraphics[scale=.75]{planes3.mps}\quad
  \includegraphics[scale=.75]{planes4.mps}\quad
  \includegraphics[scale=.75]{planes5.mps}
  \caption{Tropical polytopes in $\TT^2$ leading to generic tropical planes in $\TT^5$. Pictures
    are drawn by taking the first two coordinates in the usual directions $e_1$, $e_2$ and the last
    coordinate in direction $-e_1-e_2$.  Point labels match the lists of matroids in
    Table~\ref{tab:planes}}
  \label{fig:planes}
\end{figure}

\begin{table}\raggedright
\begin{tabular}{lll}


%
EEEG:\\
&$a$&123\ 125\ 134\ 135\ 136\ 145\ 156\ 235\ 345\ 356\\
&$b$&123\ 124\ 134\ 234\ 235\ 236\ 245\ 246\ 345\ 346\\
&$c$&123\ 124\ 134\ 136\ 146\ 235\ 236\ 245\ 246\ 345\ 346\ 356\ 456\\
&$d$&123\ 124\ 125\ 126\ 136\ 146\ 156\ 236\ 246\ 256\\
&$e$&123\ 124\ 125\ 134\ 136\ 145\ 146\ 156\ 235\ 245\ 345\ 356\ 456\\
&$f$&123\ 124\ 125\ 136\ 146\ 156\ 235\ 236\ 245\ 246\ 256\ 356\ 456\\
\addlinespace
%

EEFG:\\
&$a$&123\ 125\ 126\ 134\ 136\ 145\ 146\ 156\ 236\ 256\ 346\ 456\\
&$b$&123\ 125\ 134\ 136\ 145\ 156\ 235\ 236\ 256\ 345\ 346\ 356\ 456\\
&$c$&123\ 125\ 134\ 135\ 136\ 145\ 156\ 235\ 345\ 356\\
&$d$&123\ 124\ 134\ 234\ 235\ 236\ 245\ 246\ 345\ 346\\
&$e$&123\ 124\ 125\ 134\ 145\ 235\ 236\ 245\ 246\ 256\ 345\ 346\ 456\\
&$f$&123\ 124\ 125\ 126\ 134\ 145\ 146\ 236\ 246\ 256\ 346\ 456\\
\addlinespace
%

%
EEFF(b):\\
&$a$&123\ 124\ 134\ 234\ 235\ 236\ 245\ 246\ 345\ 346\\
&$b$&123\ 124\ 134\ 135\ 145\ 235\ 236\ 245\ 246\ 345\ 346\ 356\ 456\\
&$c$&123\ 124\ 134\ 135\ 136\ 145\ 146\ 236\ 246\ 346\ 356\ 456\\
&$d$&123\ 124\ 125\ 126\ 136\ 146\ 156\ 236\ 246\ 256\\
&$e$&123\ 124\ 125\ 135\ 145\ 235\ 236\ 245\ 246\ 256\ 356\ 456\\%
&$f$&123\ 124\ 125\ 135\ 136\ 145\ 146\ 156\ 236\ 246\ 256\ 356\ 456\\
\addlinespace
%

EFFG:\\
&$a$&123\ 125\ 126\ 136\ 156\ 234\ 236\ 245\ 246\ 256\ 346\ 456\\
&$b$&123\ 125\ 136\ 156\ 234\ 235\ 236\ 245\ 256\ 346\ 356\ 456\\
&$c$&123\ 125\ 134\ 136\ 145\ 156\ 234\ 235\ 245\ 345\ 346\ 356\ 456\\
&$d$&123\ 125\ 134\ 135\ 136\ 145\ 156\ 235\ 345\ 356\\
&$e$&123\ 124\ 125\ 134\ 136\ 145\ 146\ 156\ 234\ 245\ 346\ 456\\
&$f$&123\ 124\ 125\ 126\ 136\ 146\ 156\ 234\ 245\ 246\ 346\ 456\\
\addlinespace

FFFGG:\\
&$a$&123\ 125\ 126\ 135\ 156\ 234\ 235\ 245\ 246\ 256\ 345\ 456\\
&$b$&123\ 126\ 135\ 136\ 156\ 234\ 236\ 246\ 345\ 346\ 356\ 456\\
&$c$&123\ 126\ 135\ 156\ 234\ 235\ 236\ 246\ 256\ 345\ 356\ 456\\
&$d$&123\ 126\ 134\ 135\ 136\ 146\ 156\ 234\ 246\ 345\ 346\ 456\\
&$e$&123\ 124\ 126\ 134\ 135\ 145\ 146\ 156\ 234\ 246\ 345\ 456\\
&$f$&123\ 124\ 125\ 126\ 135\ 145\ 156\ 234\ 245\ 246\ 345\ 456
\end{tabular}
%
\caption{List of matroid bases per maximal cell for each type of generic tropical plane in $\TT^5$
  shown in Figure~\ref{fig:planes}}
\label{tab:planes}
\end{table}

\begin{proposition}\label{prop:in-Grassmannian}
  For every $V\in \RR^{k\times(n-k)}$ the point $\tau_V$ is contained in the tropical Grassmannian $\Gr kn$.
\end{proposition}

\begin{proof}
  Take any \emph{lift} of $\bar V$ to a matrix $V^*\in K^{k\times n}$ with coefficients in a Puiseux
  series field~$K$ (see \cite[Sec.~4]{TLS} and \cite{Markwig07} for details).  Then the
  valuation map $\val$ of $K$ takes $V^*_{ij}$ to $\bar V_{ij}$ for all $i$ and $j$.  The first $k$
  columns of $\bar V$ form the $k{\times}k$-tropical identity matrix, and therefore the first $k$
  columns of $V^*$ are linearly independent.  We conclude that the column space $L(V^*)$ is a
  $k$-dimensional subspace of the vector space $K^n$. Without loss of generality we may assume that
  the lift to $V^*$ is generic, so we have $\tau_V(\sigma)=\val(\frp(\sigma))$, where
  $\frp:\binom{[n]}k\to K$ are the classical Pl\"ucker coordinates of $L(V^*)$. Now
  \cite[Prop.~4.2]{TLS} implies that the tropicalization of $L(V^*)$ coincides with the
  tropical linear space defined by $\tau_V$, that is, $\tau_V$ is a tropical Pl\"ucker
  vector.
\end{proof}

\begin{remark}
  This gives us a sufficient criterion to show that a given tropical Pl\"ucker vector $\pi\in \Dr
  kn$ is actually contained in the tropical Grassmannian $\Gr kn$: For each $\sigma \in
  \binom{[n]}k$ (that is, for each vertex of $\Hypersimplex kn$) compute the matrix
  $\Phi^\sigma(\pi)$ and check if $\tau^\sigma_{\Phi^\sigma(\pi)}=\pi$. If there is some vertex
  $e_\sigma$ for which this is the case, Proposition~\ref{prop:in-Grassmannian} yields that $\pi\in
  \Gr kn$. That this criterion is not necessary follows from the existence of the generic tropical
  planes of types EEEE and EEFF(a) from Example~\ref{exmp:k=3,n=6} which cannot be obtained via this
  construction.  For EEEE this is easily seen, since the tight-span is a snowflake tree (that is, a
  tree with exactly one interior node) with six leaves.  This does not correspond to a point
  configuration in $\TT^2$.
\end{remark}

\begin{corollary}\label{cor:embedding}
  The map $\tau$ induces a piecewise-linear embedding of the secondary fan of
  $\Delta_{k-1}\times\Delta_{n-k-1}$ into the Dressian $\Dr{k}{n}$.  The image of $\tau$ is a subset
  of the tropical Grassmannian $\Gr{k}{n}$.
\end{corollary}

\begin{proof}
  It follows from Theorem~\ref{thm:tight-span} and, in particular, the construction~\eqref{eq:Sigma}
  that $\tau$ induces a piecewise-linear embedding.  The final claim is now a consequence of
  Proposition~\ref{prop:in-Grassmannian}.
\end{proof}

See also Example~\ref{exmp:k=2,cont} below.

\section{Tropically Rigid Point Configurations}\label{sec:trop-rid}

Throughout the following let $V\in\RR^{k\times(n-k)}$ be a $k\times(n-k)$-matrix with $n\ge 2k$.  We
will read (the columns of) $V$ as a configuration of $k$ labeled points in $\TT^{n-k-1}$, possibly with
repetitions.  Associated with $V$ is the regular subdivision $\Gamma$ of
$\Delta_{k-1}\times\Delta_{n-k-1}$ which is induced by lifting the vertex $(e_i,e_j)$ to $v_{ij}$.
The \emph{secondary fan} $\frS$ of $\Delta_{k-1}\times\Delta_{n-k-1}$ is the polyhedral fan in
$\RR^{k\times (n-k)}$ which arises from grouping together those lifting functions which induce the
same subdivision. The fan $\frS$ has a lineality space of dimension $n-1$.  Therefore, by taking
quotients we can view $\frS$ as a fan in $\RR^{k(n-k)-n+1}=\RR^{kn-n-k^2+1}$.

The (regular) subdivisions of $\Delta_{k-1}\times\Delta_{n-k-1}$ are partially ordered by
refinement.  The point configuration $V$ is \emph{generic} if $V$ considered as a lifting function
induces a triangulation, that is, a finest subdivision.  At the other extreme we call $V$
\emph{tropically rigid} if it induces a coarsest (non-trivial) subdivision.  The matrix $V$ gives
rise to a polyhedral subdivision of $\TT^{k-1}$ according to type, and the bounded cells form the
tropical polytope $\tconv(V)$; see Develin and Sturmfels~\cite[Sec.~3]{DevelinSturmfels04}.  The bounded
cells of the type decomposition are precisely the cells of the tight-span of $\Gamma$.  The set
$\tconv(V)$ endowed with its canonical cell decomposition by type is called the \emph{tropical
  complex} of $V$. Its vertices are the \emph{pseudo-vertices} of the point configuration $V$.  They
bijectively correspond to the maximal cells of~$\Gamma$.

The purpose of the remainder of this section is to list many examples of tropically rigid point
configurations since these will be used later to construct rays of the Dressians.

\begin{example}\label{exmp:splits}
  Consider the point
  \[
  p_\ell \ = \ (\underbrace{0,0,\dots,0}_{\ell},\underbrace{1,1,\dots,1}_{k-\ell})
  \]
  in $\TT^{k-1}$.  The tropical line segment $\tconv(0,p_\ell)$ is the ordinary line segment from
  $0$ to $p_{\ell}$.  The tropical complex has two vertices, $0$ and $p_{\ell}$, corresponding to
  the two maximal cells of the dual regular subdivision of $\Delta_{k-1}\times\Delta_1$.
  Subdivisions with precisely two maximal cells are called \emph{splits}.
\end{example}

\begin{example}\label{ex:trop-rigid}
 The tropical complex $\Gamma$ of the point configuration formed by the columns of the $k{\times}k$-matrix
  \[
  V \ = \ \begin{pmatrix}
    0      & 0 & \cdots & 0\\
    1      & 0 & \ddots & \vdots \\
    \vdots & \ddots & \ddots & 0 \\
    1      & \cdots & 1 & 0
  \end{pmatrix}
  \]
  is a $(k{-}1)$-simplex.  Its dual $\Gamma$ is a coarsest subdivision of
  $\Delta_{k-1}\times\Delta_{n-k-1}$~(more precisely, a \emph{$k$-split}) and hence the point
  configuration $V$ is tropically rigid, see \cite[Sec.~4]{Herrmann09}).
\end{example}

For a subdivision $\Gamma$ of $\Delta_{k-1}\times\Delta_{n-k-1}$, one can give an explicit
description of finitely many inequalities and equations describing the secondary cone of of all
weight functions functions yielding this subdivision; see \cite[Cor.~5.2.7]{Triangulations}.  So,
given a matrix $V\in\RR^{k\times(n-k)}$ we can decide if the subdivision $\Gamma$ it describes is
coarsest by solving a linear program to determine the dimension of the secondary cone of $\Gamma$.

\begin{example}\label{exmp:rigid-special}
  The columns of the matrix
  \[
  V \ = \ \begin{pmatrix}
    0 & 0 & 0 &  0 &  0\\
    1 & 1 & 0 & -1 & -1\\
    0 & 1 & 1 &  0 & -1
  \end{pmatrix} \, ,
  \]
  form a configuration of five points in $\TT^2$, see Figure~\ref{fig:rigid-special} to the right.
  The origin $(0,0,0)$ is a pseudo-vertex, but it does not occur among the generators.  The induced
  regular subdivision of $\Delta_2\times\Delta_4$ is coarsest, and $V$ is tropically rigid.  The
  tropical Pl\"ucker vector $\tau_V$ induces a coarsest matroid subdivision of
  $\Hypersimplex{3}{8}$.
\end{example}

\begin{figure}[htb]
  \includegraphics{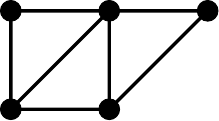}\qquad
  \includegraphics{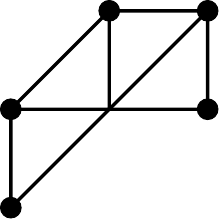}
  \caption{Two tropically rigid configurations of five points in $\TT^2$ without multiple points}
  \label{fig:rigid-special}
\end{figure}

Duplicating points in a tropically rigid point configuration is a general way to create new
tropically rigid point configurations from old ones, as the next result shows.

\begin{proposition}\label{prop:duplicate}
  Let $V\in\RR^{k\times(n-k)}$ be a tropically rigid point configuration, and $w$ any column
  of $V$.  Then the point configuration $V':=(V|w)\in\RR^{k\times(n+1-k)}$ is also tropically
  rigid.
\end{proposition}

\begin{proof}
  Let $\Gamma$ be the subdivision of $\Delta_{k-1}\times\Delta_{n-k-1}$ induced by $V$ and $\Gamma'$
  the subdivision of $\Delta_{k-1}\times\Delta_{n-k}$ induced by $V'$.  Suppose $w$ is the $i$-th
  column of $V$.  By looking at the lifted polytopes, it follows that the maximal cells of $\Gamma'$
  are exactly the polytopes
  \[
  \conv\bigl(\smallSetOf{e_\ell+e_m}{e_\ell+e_m \in C} \cup \smallSetOf{e_\ell+e_{n+1}}{e_\ell+e_i\in C}\bigr) \, ,
  \]
  where $C$ is a maximal cell of $\Gamma$.  In particular, $\Atightspan \Gamma$ and
  $\Atightspan{\Gamma'}$ are isomorphic.  The equation $x_{n+1}=0$ defines a facet $F$ of
  $\Delta_{k-1}\times\Delta_{n-k}$ which is isomorphic to $\Delta_{k-1}\times\Delta_{n-k-1}$.  The
  subdivision $\Gamma'|_F$ of $F$ induced by $\Gamma'$ is isomorphic to $\Gamma$, which is a
  coarsest subdivision by assumption. Hence the tight-spans of $\Gamma$ and $\Gamma'$ agree.
  Applying Proposition~\ref{prop:ts-isom}~\eqref{prop:ts-isom:coarsest} to the subpolytope $F$ of
  $\Delta_{k-1}\times\Delta_{n-k}$ shows that $\Gamma'$ is a coarsest subdivision, too.  This means
  that the point configuration $V'$ is tropically rigid.
\end{proof}

Figure~\ref{fig:trop-rigid} below shows tropically rigid configurations of five points in the plane
which arise from configurations with fewer points by duplicating as described in
Proposition~\ref{prop:duplicate}.

\begin{figure}[hbt]
  \includegraphics{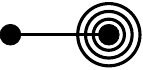}\qquad
  \includegraphics{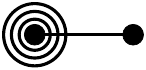}\qquad
  \includegraphics{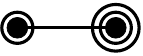}\qquad
  \includegraphics{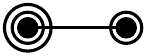}\\[0.5cm]
  \includegraphics{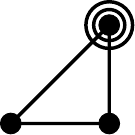}\qquad
  \includegraphics{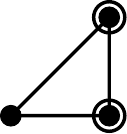}\qquad
  \includegraphics{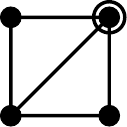}\qquad
  \includegraphics{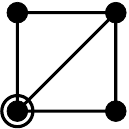}\qquad
  \includegraphics{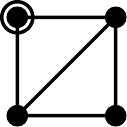}
  \caption{Tropically rigid configurations of five points in~$\TT^2$ with multiple points; multiple
    points circled}
  \label{fig:trop-rigid}
\end{figure}

\begin{remark}
  Up to equivalence, there are exactly eleven tropically rigid configurations of five points in
  $\TT^2$.  The only two configurations without multiple points are shown in
  Figure~\ref{fig:rigid-special}.  In Figure~\ref{fig:trop-rigid} we depict the ones with at least
  one multiple point.  Notice that all four splits in the upper row are pairwise not equivalent as
  point configurations in $\TT^2$.
  
  That this, indeed, is the complete list follows from the classification of the tight-spans of rays
  of $\Dr 3n$ in Example~\ref{ex:DR38-rays} below in connection with
  Theorem~\ref{thm:tight-span}.
\end{remark}

\begin{example}\label{exmp:k=2,cont}
  Continuing our Example~\ref{exmp:k=2} here again we consider the case $k=2$. All rays of $\Dr 2n$
  correspond to splits of $\Hypersimplex 2n$, and these also give the only tropically rigid point
  configurations (without duplicates) in this case.

  The $(n{-}3)$-dimensional permutahedron is a secondary polytope of $\Delta_1\times\Delta_{n-3}$;
  see \cite[\S6.2.1]{Triangulations}.  That is, its normal fan is the secondary fan of that product
  of simplices.  This secondary fan embeds into the Dressian $\Dr 2n$, which coincides with $\Gr
  2n$, as described in Corollary~\ref{cor:embedding}.  For $n=5$ the Dressian $\Dr 25$ (as a
  spherical polytopal complex) is isomorphic to the Petersen graph (considered as a $1$-dimensional
  spherical polytopal complex).  The Dressian $\Dr 25$ is also the graph complement of the the
  vertex-edge graph of $\Hypersimplex 25$, and thus we can label its vertices with $2$-element
  subsets of the set $\{1,2,3,4,5\}$.  This comes about as each ray of the secondary fan of
  $\Hypersimplex 25$ is a vertex split; for instance, this follows from the classification of
  hypersimplex splits \cite[\S5]{MR2502496}.  The $2$-dimensional permutahedron (as well as its
  dual) is a hexagon.  The embedding with respect to the vertex $e_{12}$ is shown in
  Figure~\ref{fig:Dr25}.
\end{example}

\begin{figure}[hbt]
  \includegraphics[scale=1]{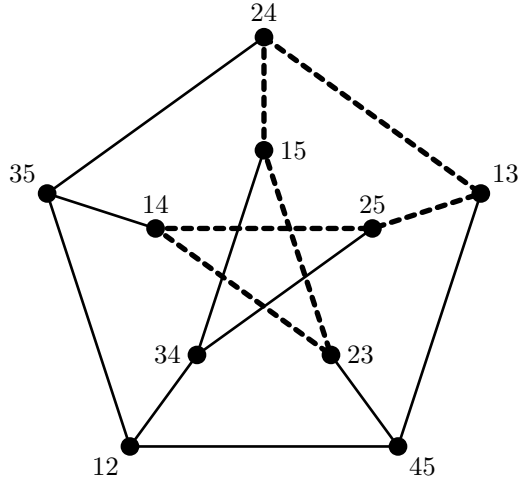}
  \caption{$\Dr 25$ with embedded hexagon}
  \label{fig:Dr25}
\end{figure}

\section{Combinatorial Properties of Coarsest Matroid Subdivisions}

For the relevant definitions and basic properties of matroids we refer to White
\cite{White86,White92}.  Let $\cM$ be a matroid on the set $[n]$.  Two elements $i,j\in[n]$ are said
to be \emph{equivalent} if there exists a circuit $C$ of $\cM$ with $i,j\in C$. The equivalence
classes of this relation are the \emph{connected components} of $\cM$. We denote by $c(\cM)$ the
number of connected components of $\cM$. A matroid is called \emph{connected} if it has
$c(\cM)=1$. In fact, there is a relation between the number of connected components of a matroid and
the dimension of its matroid polytope.

\begin{proposition}[{Feichtner and Sturmfels~\cite[Prop.~2.4]{MR2191630}}]\label{prop:dimension}
  Let $\cM$ be a matroid with $n$ elements. Then the
  dimension of the matroid polytope of $\cM$ equals $n-c(\cM)$.
\end{proposition}

If a connected component of $\cM$ has cardinality $1$, it is called \emph{trivial} and its unique
element is a \emph{loop} of $\cM$. Since bases are maximal with the property of not containing a
circuit, it follows that a basis of $\cM$ has to contain at least one element from each non-trivial
connected component of $\cM$. This says that the number of non-trivial connected components of a
matroid is bounded by its rank.
 
The restriction of $\cM$ to each of its connected components $C_1,C_2,\dots,C_\nu$ is a matroid
$\cM(C_i)$ with ground set $C_i$ and, obviously, one has
\[
\sum_{i=1}^{\nu} \rank \cM(C_\nu)=\rank \cM.
\]
This is also true for matroids with trivial components, as a trivial component has rank $0$.
 
Let now $\Sigma$ be a matroid subdivision of $\Hypersimplex kn$ and $F\in\Sigma$ some interior cell.
Then there exists a rank-$k$ matroid $\cM$ with $n$ elements such that the matroid polytope of $\cM$
is $F$.  Furthermore ($F$ being interior) $\cM$ does not contain any loops so it has at most $k$
connected components.  Since the number of connected components of a loop-free matroid is at most
the rank, translating Proposition~\ref{prop:dimension} to the language of tight-spans gives us the
following:

\begin{lemma}\label{lem:dimension}
  Tight-spans of matroid subdivisions of $\Hypersimplex kn$ are at most $(k{-}1)$-dimensional.
\end{lemma}

In particular, if $F$ has codimension $k-1$, $\cM$ has $k$ connected components $C_1,C_2,\dots,C_k$ and
the bases of $\cM$ are the sets containing exactly one element from each $C_i$. Otherwise stated,
$\cM$ is the product of $k$ rank-$1$ matroids. If $F$ has codimension $1$, its matroid $\cM$ has two
connected components $C_1,C_2$, $F$ lies in the hyperplane $\sum_{i\in C_1}x_i=\rank \cM(C_1)=:l$,
and $\cM$ is the product of a rank-$l$ and a rank-$(k-l)$ matroid.

We will now examine the rays of $\Dr kn$. By Lemma~\ref{lem:dimension}, these correspond to
$k$-dimensional tropical linear spaces. However, in this least generic case we also have to deal
with lower dimensional degenerations up to embeddings of tropical lines; these arise as
degenerations of proper (tropical) planes.

\begin{proposition}[{\cite[Prop.~3.4]{MR2515769}}]
  Each split of the hypersimplex $\Delta(k,n)$ gives a ray of the Dressian $\Dr kn$.
\end{proposition}

This holds true since such a split is a regular matroid subdivision with precisely two maximal
cells; so this must be a coarsest subdivision.  Via tropical rigidity the same result can also be
obtained as follows: Applying Proposition~\ref{prop:duplicate} to Example~\ref{exmp:splits} gives us
a procedure to generate splits of any hypersimplex.  In fact, all hypersimplex splits arise via
picking a vertex (figure), the parameter $\ell$ (to determine the point $p_\ell$), and the
duplication pattern.  By \cite[Thm.~5.3]{MR2502496} the total number of 
splits of the hypersimplex $\Hypersimplex{k}{n}$ with $n\ge 2k$ equals
\begin{equation}
  \label{eq:n_splits}
  (k-1) \left(2^{n-1}-(n+1)\right)-\sum_{i=2}^{k-1}(k-i)\binom{n}{i} \, .
\end{equation}

Further known rays of $\Dr kn$ are the $3$-splits (subdivisions whose tight-span is a triangle) as
shown in \cite[Thm.~6.5]{Herrmann09}. Keeping the same notation as above, we now specialize to the
case $k=3$.  This means that we are looking at tropical point configurations in $\TT^2$ and at the
induced tropical planes in $\TT^{n-1}$.  For the number of splits of $\Delta(3,n)$ with $n\ge 6$ the
formula~\eqref{eq:n_splits} reads
\[
  (3-1) \left(2^{n-1}-(n+1)\right)-\sum_{i=2}^{3-1}(3-i)\binom{n}{i} \ = \
  2^n - \frac{1}{2} \left(n^2+3n+4\right) \, .
\]

The following is readily implied by Corollary~\ref{cor:embedding}.

\begin{corollary}\label{cor:trop-rigid-ray}
  Any tropically rigid configuration $V$ of $n$ points in $\TT^{k-1}$ gives rise to a ray of the
  Dressian $\Dr kn$, which is also a ray of the tropical Grassmannian $\Gr kn$, such that the
  tight-span of the ray coincides with the tropical complex of $V$.
\end{corollary}

\begin{corollary}\label{cor:trop-rigid}
  For $n\ge 2k$ there are at least $T(n-k,k)$ combinatorially distinct rays of $\Dr kn$ and $\Gr kn$.
\end{corollary}
Here $T(m,k)$ is the number of partitions of $m$ into $k$ positive parts, which is the same as the
number of partitions of $m$ in which the greatest part is $k$; see Integer Sequence: A008284 \cite{IntSeq}. We
have the recursion
\[
T(m, k) \ = \ \sum_{i=1}^k T(m-k, i)
\]
with $T(m, m)=1$.

\begin{proof}
  Consider the tropical complex $\Gamma$ from Example~\ref{ex:trop-rigid} which defines a $(k-1)$-dimensional
  simplex. We can now distribute $n-2k$ additional points arbitrarily among the $k$ vertices of $\Gamma$, creating
  multiple points this way. 
  Two such configurations are equivalent under the action of the group $\Sym(k)$ if and only if they
  correspond to the same partition of $n-k$ into $k$ positive parts. By Proposition~\ref{prop:duplicate} the new
  configuration is tropically rigid and, by Corollary~\ref{cor:trop-rigid-ray}, corresponds to a ray of the
  $\Dr kn$ and $\Gr kn$.
\end{proof}

\begin{remark}
  For $n\ge 6$ we have
  \[
  T(n-3,3) \ = \ \lfloor 1/12 (n-3)^2 + 1/2 \rfloor \, ;
  \]
  see Integer Sequence A001399 \cite{IntSeq}.  This establishes a quadratic lower bound for the number of
  combinatorially distinct configurations of $n-3$ points in $\TT^2$ which are tropically rigid.
\end{remark}

\section{Most Degenerate Tropical Planes}
\label{sec:planes}

We will now concentrate on the case $k=3$ and discuss the coarsest matroid subdivisions of the
hypersimplices $\Delta(3,n)$.  These are precisely the rays of the Dressian $\Dr 3n$.  Generically,
these correspond to tropical planes.  In view of Lemma~\ref{lem:dimension}, we can give the
following restriction on the tight-span of a coarsest matroid subdivision of $\Dr 3n$.

\begin{proposition}\label{prop:dimension-3}
  Let $\subdiv$ be a coarsest matroid subdivision of $\Dr 3n$. Then $\Atightspan \Sigma$ is
  either a line segment or pure two-dimensional.
\end{proposition}

\begin{proof}
  By Lemma~\ref{lem:dimension}, $\Atightspan \Sigma$ is at most two-dimensional. If $\subdiv$
  is a split,  $\Atightspan \Sigma$ is a line segment; otherwise,
  \cite[Cor.~5.3]{Herrmann09} implies that all maximal faces of $\Atightspan \Sigma$ are
  at least two-dimensional. This shows the claim.
\end{proof}

Once again specializing our discussion in the previous section to the case $k=3$, it follows that a
matroid $\cM$ corresponding to a codimension-$1$ cell of a matroid subdivision of $\Hypersimplex 3n$
is the product of a rank-$2$ matroid and a rank-$1$ matroid. So, in particular, $\cM$ is realizable
over any infinite field $\KK$. 

The matroid $\cM(C)$ of a codimension-$2$ cell $C$ is even simpler.  Such a matroid must be a rank
$3$ matroid with no loops and three connected components.  Let the three components be $\{ \alpha,
\beta, \gamma \}$, with $[n] = \alpha \sqcup \beta \sqcup \gamma$.  So the bases of $\cM(c)$ are the
triples $(i,j, \ell)$ with $i \in \alpha$, $j \in \beta$, $\ell \in \gamma$.  Every codimension-$1$
cell $F$ containing~$C$ lies in one of the three hyperplanes $H_{\alpha}$, $H_{\beta}$ or
$H_{\gamma}$.  So $C$ is contained in at most six codimension-$1$ cells (each hyperplane is divided
in half by $C$, and can contribute a cell from each side).

Let $S$ be the two-dimensional cell of the tight-span dual to $C$. So $S$ has at most $6$ sides.  If
the subdivision of $\Delta(3,n)$ is regular, so that $S$ comes with an embedding into
$\mathbb{R}^n/\mathbb{R} \cdot (1,\ldots, 1)$, then the edges of $C$ are orthogonal to $H_{\alpha}$,
$H_{\beta}$ and $H_{\gamma}$, meaning that they are parallel to the vectors $\sum_{i \in \alpha}
e_i$, $\sum_{i \in \beta} e_i$ and $\sum_{i \in \gamma} e_i$. Note that, in $\mathbb{R}^n/\mathbb{R}
\cdot (1,\ldots, 1)$, these three vectors sum to $0$.  So $C$ is either a triangle, a parallelogram,
a trapezoid, a pentagon or a hexagon. (See Figure~\ref{fig:xgons}.)  In fact, for any rank $k$,
these are the only possible $2$-faces in the tight-span of a tropical linear
space. See~\cite[Prop. 2.6]{TLS}.

\begin{figure}[hbt]
  \includegraphics[scale=0.9]{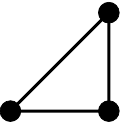}\qquad
  \includegraphics[scale=0.9]{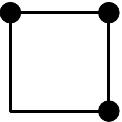}\qquad
  \includegraphics[scale=0.9]{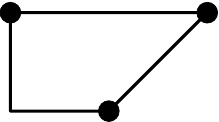}\qquad
  \includegraphics[scale=0.9]{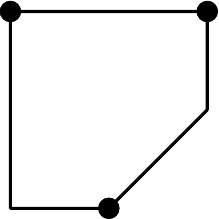}\qquad
  \includegraphics[scale=0.9]{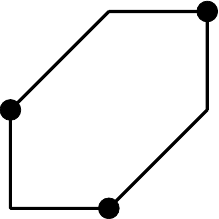}
  \caption{Shapes that might appear in the tight-span of a matroid subdivision; these precisely are the two-dimensional polytropes~\cite{JK10}}
  \label{fig:xgons}
\end{figure}

The vertices of the tight-span are in bijection with the maximal cells of the matroid subdivision. 
When we coarsen a matroid subdivision, adjacent cells merge into larger cells; the dual effect is that edges of the tight-span contract down to length $0$.
We will say that such an edge \emph{collapses}.
We make the following observations:

\begin{proposition} \label{CollapseTriangle}
If one edge of a triangle collapses, then all the edges of that triangle collapse.
\end{proposition}

\begin{proposition} \label{CollapseParallelogram}
If one edge of a parallelogram collapses, so does the opposite edge.
\end{proposition}

These conditions make it possible in many cases to recognize that a subdivision cannot be
non-trivially coarsened.  For example, if any edge in Figure~\ref{fig:non-planar} is collapsed, then
all $7$ edges must be. Collapsing all the edges gives the trivial subdivision.  So the tight-span in
Figure~\ref{fig:non-planar} cannot be non-trivially coarsened.

\begin{figure}[hbt]
  \includegraphics{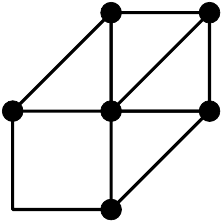}
  \caption{Configuration of six points corresponding to a coarsest tropical subdivision that contains a quadrangle}
  \label{fig:4gon}
\end{figure}

In a preprint version of this paper, it was claimed that coarsest subdivisions only contain
triangles, not faces with higher numbers of edges.  We can use the above propositions, together with
the methods of Section~\ref{sec:Polytope}, to refute this claim: Figure~\ref{fig:4gon} shows six
points in $\mathbb{R}^2$ whose tropical convex hull is a hexagon.  The natural polyhedral
subdivision of this hexagon is into four triangles and a parallelogram.  The corresponding tropical
$3$-plane in $9$-space thus also has a tight-span which consists of four triangles and a
parallelogram arranged around a central point.  Again, Propositions~\ref{CollapseTriangle}
and~\ref{CollapseParallelogram} show that, if we collapse any edge, we must collapse all of them.
So this subdivision is coarsest, even though it has a quadrilateral face.

\begin{example}\label{ex:DR38-rays}
  There are $15,\!470$ rays of $\Dr 38$ coming in twelve symmetry classes. We begin with listing the
  ones that were known before:
  \begin{itemize}
  \item The simplest ones are the splits. According to \eqref{eq:n_splits} the total number of
    splits of $\Hypersimplex{3}{8}$ equals $210$. This gives four distinct splits up to symmetry, in
    the notation of~\cite{MR2502496} these correspond to the $(6,2;1)$-, $(3,5;1)$-, $(5,3;1)$-, and
    $(4,4;1)$-hyperplanes.  We obtain $28$, $56$, $56$ and $70$ rays per orbit, respectively. The
    corresponding tropically rigid configuration are shown in the top part of
    Figure~\ref{fig:trop-rigid}.
  \item Furthermore, in \cite[Sec.~6]{Herrmann09} it is shown that certain \emph{$3$-splits},
    that is, coarsest subdivisions with three maximal faces, (or, equivalently, subdivisions, whose
    tight-span is a triangle) correspond to rays of $\Dr kn$. Specifically,
    \cite[Cor.~6.4]{Herrmann09} tells us that there are $980$ of these $3$-splits coming in two
    equivalence classes. These are also obtained by the two tropically rigid configuration in
    Figure~\ref{fig:trop-rigid}, where the left one gives $420$ and the right one $560$ rays.
  \end{itemize}
  Further rays are given by our tropically rigid point configurations of Section~\ref{sec:trop-rid}:
  \begin{itemize}
  \item The next tight-span that might occur are two triangles connected by an edge. The
    corresponding subdivisions come in three equivalence classes and may be obtained by the three
    tropical point configurations to the right of Figure~\ref{fig:trop-rigid}. They give us $840$,
    $1,\!260$ and $1,\!680$ rays, respectively, yielding a total of $3,\!780$ with such a
    tight-span.
  \item The remaining rays of the Dressian coming from tropical point configurations are those
    depicted in Figure~\ref{fig:rigid-special}, hence corresponding to three or four connected
    triangles. Both of these give rise $5,\!040$ rays.
  \end{itemize}

  However, not all tight-spans of rays of $\Dr 3n$ need to be planar, hence not all rays may be
  induced by tropically rigid point configurations. Indeed, our computations show that the simplest
  non-planar simplicial complex occurs as the tight-span of a ray of $\Dr 38$. This tight-span is
  depicted in Figure~\ref{fig:non-planar} and it gives rise to the remaining $420$ rays.

  Altogether, Figures~\ref{fig:rigid-special},~\ref{fig:trop-rigid} and~\ref{fig:non-planar} give
  the tight-spans of all rays of $\Dr38$. Explicit coordinates for the first eleven rays can be
  obtained by applying the map $\tau$ to the tropical point configurations in
  Figures~\ref{fig:rigid-special} and \ref{fig:trop-rigid}.
  The last ray is a $0/1$-vector of length $\tbinom{8}{3}=56$ which maps the vertices of
  $\Delta(3,8)$ corresponding to the following $30$ three-element subsets of $\{1,2,\dots,8\}$ to $0$, and the
  $26$ remaining ones to~$1$:
  \begin{align*}
    123\ 124\ 126\ 127\ 128\ 134\ 136\ 137\ 138\ 234\ 235\ 236\ 237\ 238\ 245\ \\
    247\ 248\ 256\ 257\ 258\ 267\ 268\ 345\ 347\ 348\ 356\ 357\ 358\ 367\ 368 \, .
  \end{align*}
  These sets form the bases of a matroid.
\end{example}

\begin{figure}[htb]
  \includegraphics{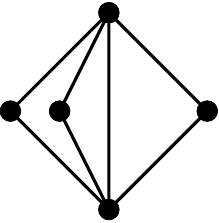}
  \caption{Non-planar tight-span of a ray of $\Dr 38$.  This is a $2$-dimensional pure simplicial complex of three triangles sharing a common edge.}
  \label{fig:non-planar}
\end{figure}

\section{Computational Results}
\label{sec:computation}

Here we explain how we computed the entire fan $\Dr{3}{8}$ with
\texttt{polymake}~\cite{DMV:polymake}.  Fundamentally, we used a similar approach as for computing
$\Dr{3}{7}$ in \cite{MR2515769}.  However, we re-implemented that algorithm with a number of
modifications.  The key new idea is to have a short canonical description for each matroid
subdivision such that it is fast to recognize duplicates.  For $\Dr{3}{7}$ this gives a speed up
factor of about four.  At the same time, our new approach is more efficient with respect to memory
consumption.  We verified the computational results for $\Dr 36$ from \cite{MR2071813} and
$\Dr{3}{7}$ from \cite{MR2515769} with our new implementation.

A cone of the polyhedral fan $\Dr kn$ is defined by deciding for each $3$-term Pl\"ucker
relation~\eqref{eq:3term} which two of the three terms attain the maximum.  This is to say, such a
cone is defined by equations and inequalities of the form
\begin{align}\label{eq:trop3term}
  \pi(\rho ij)+\pi(\rho\ell m) \ = \ \pi(\rho i\ell)+\pi(\rho jm) \ \ge \ \pi(\rho im)+\pi(\rho j\ell)
\end{align}
where $\rho$ is a subset of $[n]$ of cardinality $k-2$ and $i,j,\ell,m\in[n]\setminus \rho$ are
pairwise distinct.  Enumerating all maximal cones in $\Dr{k}{n}$ is now equivalent to going through
all possible $3^\nu$ combinations, where $\nu=\tbinom{n}{4} \cdot \tbinom {n-4}{k-2}=\tbinom{n}{k-2}
\cdot \tbinom{n-k+2}{4}$ is the number of $3$-term Pl\"ucker relations.  This direct but na\"ive
approach is, of course, infeasible in terms of complexity.

For $\rho\in\tbinom{[n]}{k-2}$ and $i,j,\ell,m\in[n]\setminus \rho$ the six points
\[
e_{\rho ij} \, , \quad e_{\rho i\ell} \, , \quad e_{\rho im} \, , \quad e_{\rho j\ell} \, , \quad
e_{\rho jm} \, , \quad e_{\rho \ell m}
\]
are the vertices of a regular octahedron $O$, which forms a $3$-face of the hypersimplex
$\Delta(k,n)$.  The three choices to pick the maximum in \eqref{eq:trop3term} correspond to the
three ways to split $O$ into two square pyramids.  These are called \emph{octahedral splits}.  Each
octahedral split gives one equation and one inequality.

Geometrically, the cones of $\Dr kn$ correspond to matroid decompositions of $\Hypersimplex{k}{n}$.  The
three-dimensional faces of $\Hypersimplex{k}{n}$ are either simplices or octahedra such as $O$.  Each matroid
decomposition of $\Hypersimplex kn$ induces a matroid decomposition on each of its faces.  Simplices do not
admit any non-trivial subdivision (without new vertices).  So only the three splits of the regular
octahedron give rise to a non-trivial matroid subdivisions of any three-dimensional matroid polytope.
This way, each cone of $\Dr kn$ or, equivalently, each matroid subdivision of $\Hypersimplex{k}{n}$ yields
splits on a subset of its octahedral $3$-faces.  We encode this information as a \emph{sequence of
  octahedral splits}: for a fixed linear ordering of all octahedral faces of $\Hypersimplex{k}{n}$ we list
if that face is subdivided or not, and if so, by which of the three possible splits.  The amount of
storage required is two bits per octahedral $3$-face, and this totals to $2\nu$ bits.

A slight variation of the na\"ive algorithm, conceptually, now leads to a first backtracking scheme
to compute $\Dr{k}{n}$ with the fan structure imposed by the Pl\"ucker relations.

\begin{enumerate}
\item Start with the entire space $C=\RR^{\tbinom nk}$ and the empty sequence $L$ of
    octahedral splits.
\item Iterate through all octahedral faces.
\item For each new octahedron, add one of the three
    possible splits to $L$ and intersect $C$ by the corresponding hyperplane and halfspace.
\item If $\dim C$ is too small for a maximal cone, backtrack and try another possibility
    for the same octahedron.
\item If all three splits for an octahedron have been tried, backtrack and try another
    octahedron.
\item If we are at the last octahedron, output $C$ and $L$ continue backtracking.
\end{enumerate}

The output will be a list of all maximal cones of $\Dr{k}{n}$ along with their description as
sequences of octahedral splits. However, due to the vast amount of maximal cones in $\Dr{3}{8}$, the
computation is still infeasible. Therefore, we here give the following modified algorithm, using the
same idea, but computing $\Dr{k}{n}$ iteratively from $\Dr{k}{n-1}$ and taking the known boundary
into account.  This works since each facet of $\Hypersimplex{k}{n}$ is again a hypersimplex, either of type
$\Delta(k-1,n)$ or of type $\Delta(k,n-1)$.  Moreover, each sequence of octahedral splits of a
hypersimplex induces a sequence of octahedral splits on each facet (and thus, inductively, on each
lower dimensional face).

\begin{enumerate}
\item After arriving at a new octahedral face and adding a split, consider the octahedral splits
  induced in each $\Hypersimplex{k}{n-1}$ boundary face of $\Hypersimplex{k}{n}$.
\item Test if there is some matroid subdivision of $\Hypersimplex{k}{{n-1}}$ defining these octahedral splits.
\item If this is not the case go back and try another of the three possibilities to split the octahedron.
\end{enumerate}

With this algorithm we were able to compute all maximal cones of $\Dr{3}{7}$ in around 15 minutes (from
$\Dr 36$ which is computed with either algorithm in a few seconds).  All timings are taken
single-threaded on a machine with AMD Athlon(tm) 64 X2 Dual Core Processor 4200+ (4433.05 bogomips
per core) 4~GB main memory, running Ubuntu 10.04 (Lucid Lynx).  To compute $\Dr 38$, we made the
further modification to assume that for the first Pl\"ucker relations one of the three possibilities
is fixed. In this way, the algorithm does not compute all maximal cones any more, but it is still
guaranteed that we get at least one maximal cone in each symmetry class.  With this reduction, the
computation for $\Dr 38$ took about 200 hours.  The output of this algorithm, however, has to be
processed further to yield anything useful.

To get a complete description of $\Dr 38$ from this result, we first produced a list containing one
representative of each symmetry class of maximal cones together with the corresponding orbit
size. This is obtained by the following algorithm.

\begin{enumerate}
\item Initialize $L$ as the empty list.
\item For each cone compute the lexicographic first cone $C$ in the same orbit.
\item If $C$ is in $L$ proceed to the next cone, otherwise add $C$ to $L$ and compute the orbit of
  $C$ and store the size.
\end{enumerate}

This computation took around 230 hours.  Of course, it would be faster to store all cones
from all orbits during the computation, however this not feasible in terms of space.

With this list of maximal cones the following further steps were necessary to compute the
$f$-vector and the $f$-vector up to symmetry:

\begin{enumerate}
\item For one maximal cone from each orbit, compute the rays.
\item Compute one ray form each orbit and then all rays and store them in a list~$R$.
\item Each maximal cone is no translated in a description by the indices of its rays in~$R$.
\item For one maximal cone in each orbit compute all faces of a fixed dimension $d$, and represent
  it by the indices of its rays in~$R$.
\item For each possible dimension $d$, go through all faces so computed and compute one from each
  orbit together with the orbit size similar as above.
\end{enumerate}

For dimension $6$, the computations took about 14.5 hours, this was the maximum.  The final result can
be summed up as follows.

\begin{theorem}\label{thm:polymake}
  The Dressian $\Dr38$ is a non-pure non-simplicial nine-dimensional polyhedral fan with $f$-vector
  \begin{equation*}
    \begin{aligned}
      (1;&\,15,\!470;\,642,\!677;\,8,\!892,\!898 ;\,57,\!394,\!505 ;\,194,\!258,\!750
      ; \\
      &\,353,\!149,\!650 ;\,324,\!404,\!880;\,117,\!594,\!645 ;\,113,\!400 )\,.
    \end{aligned}
  \end{equation*}
  Modulo the natural $\Sym(8)$-symmetry, the $f$-vector reads
  \[
  (1;\,12;\,155;\, 1,\!149;\, 5,\!013;\,12,\!737 ;\,18,\!802 ;\,14,\!727 ;\,4,\!788;\,14)\,.
  \]
  There are $116,\!962,\!265$ maximal cones, $113,\!400$ of dimension $9$ and
  $116,\!848,\!865$ of dimension $8$. Up to symmetry, there are $4,\!748$ maximal cones,
  $14$ of dimension $9$ and $4,\!734$ of dimension $8$.
\end{theorem}

\begin{corollary}
  There are $4,\!748$ distinct combinatorial types of generic tropical planes in~$\TT^6$.
\end{corollary}

\begin{proposition}
  All rays of $\Dr38$ are rays of $\Gr38$.
\end{proposition}

\begin{proof}
  By Proposition~\ref{prop:in-Grassmannian} all rays coming from tropically rigid point
  configuration are elements of the Grassmannian. So it remains to show that the last ray is
  contained in $\Gr 38$. This is done by explicit computation with the computer algebra
  software \texttt{Macaulay2}~\cite{M2}. We computed the initial ideal $I_r$ of the Pl\"ucker ideal
  defined by the weight vector corresponding to this ray and verified that $I_r$ does not
  contain any monomials. Note, however, that the tropical Grassmannian depends on the characteristic
  of the field considered.  Therefore, we computed over the polynomial ring $\ZZ[p_S]$ with integer
  coefficients.  It turns out that all polynomials in the integral Gr\"obner basis of $I_r$ only
  have non-vanishing coefficients $\pm1$, yielding the result for arbitrary characteristic.
\end{proof}

The data computed as available at \url{http://svenherrmann.net/DR38/dr38.html}.

\section{Questions}

We would like to close this paper with some open questions.

\begin{question}
  Can all coarsest matroid subdivisions of $\Delta(3,n)$ with planar tight-spans be constructed via
  tropical point configurations?
\end{question}

Our computations show that this question has a positive answer for all $n\le 8$.

\medskip

There is a very crude parameter to estimate how complicated a subdivision is from the combinatorial
point of view.  The \emph{spread} of an element of $\Dr kn$ is defined in as the number of maximal
cells of the corresponding subdivision. It was shown in \cite[Prop.~3.7]{MR2515769} that the spread
is not bounded if $n$ increases (with $k$ fixed).

\begin{question}
  What is the maximal \emph{spread} of a ray of $\Dr kn$ (at least for $k=3$)?
\end{question}
 
For $k=3$ and $n=5,6,7,8$, the values are $2,3,4,6$, respectively.

\medskip

It is known that the natural fan structures of the tropical Grassmannians $\Gr kn$ and the Dressians
$\Dr kn$ differ for $k\ge 3$.  However, the following is unclear to us.

\begin{question}
  Are the rays of $\Dr kn$ always rays of $\Gr kn$?
\end{question}

This was known to be true previously for $k=3$ and $n\leq 7$.  It follows from our results in
Section~\ref{sec:computation} that it also holds for $k=3$ and $n=8$.  The converse is not true:
there are rays of $\Gr 37$ that are not rays of $\Dr{3}{7}$.

\medskip

There is also a number of obvious computational challenges.

\begin{question}
  Is it feasible to compute $\Gr{3}{8}$?
\end{question}

With currently available implementations of the relevant algorithms, such as
\texttt{Gfan}~\cite{gfan}, the answer seems to be \emph{no}, even for some fixed characteristic.  It
should be feasible, however, to sample (few) points from each maximal face of the Dressian
$\Dr{3}{8}$ and check if they are contained in the tropical Grassmannian of a fixed characteristic.
Doing the necessary Gr\"obner bases computations over the integers, to obtain results valid for all
characteristics, is more demanding.

Even if we cannot tell exactly how $\Dr{3}{9}$, $\Dr{3}{10}$, etc.\ look like, by now we have
gathered substantial information about the Dressians $\Dr{3}{n}$ for arbitrary $n$.  However, it is
expected that entirely new features arise if one replaces the first parameter by $4$.  So another
very serious challenge is to answer the following.

\begin{question}
  Is it feasible to compute $\Dr{4}{8}$?
\end{question}

\bibliographystyle{amsplain}
\bibliography{main}

\end{document}